\documentclass{amsart}

\usepackage{amssymb} 
\usepackage{amsmath} 
\usepackage{latexsym} 
\usepackage{amsfonts}
\usepackage{amsthm}  
\usepackage{mathrsfs}
\usepackage{verbatim} 
\usepackage{stmaryrd} 
\usepackage{color}  
\usepackage[pdftex,colorlinks=true, pdfstartview=FitV, linkcolor=blue, citecolor=blue, urlcolor=blue]{hyperref} 
\usepackage[all,2cell,ps]{xy}
\usepackage{tikz}
\usepackage{lscape} 

\numberwithin{equation}{section}

\newcommand{\ds}{\displaystyle}
\newcommand{\ol}{\overline}
\newcommand{\ul}{\underline}
\newcommand{\vl}{\;\vert\;}

\newcommand{\Z}{\mathbb{Z}}

\newcommand{\C}{\mathbb{C}}
\newcommand{\wh}{\widehat}
\newcommand{\m}{\mathfrak{m}}

\newcommand{\x}{\mathbf{x}}

\newcommand{\Tor}{\mathrm{Tor}}
\newcommand{\gr}[2]{\mathrm{gr}_{#1} #2}

\newcommand{\ann}{\mathrm{ann}}

\newcommand{\im}{\mathrm{im}}
\newcommand{\CM}{Cohen-Macaulay\ }
\newcommand{\CMp}{Cohen-Macaulay}
\newcommand{\ls}{\leqslant}
\newcommand{\gs}{\geqslant}

\newcommand{\SuS}{super-stretched\ }
\newcommand{\SuSp}{super-stretched}

\newcommand{\lst}[2]{#1_1,\dots,#1_{#2}}

\newcommand{\MCM}[1]{\mathfrak{MCM}(#1)}
\newcommand{\MCMgr}[1]{\mathfrak{MCM}^{\mathrm{gr}}(#1)}
\newcommand{\Mod}[1]{\mathfrak{mod}(#1)}
\newcommand{\Modgr}[1]{\mathfrak{mod}^{\mathrm{gr}}(#1)}

\renewcommand{\mod}[1]{\ (\mathrm{mod}\ #1 )}

\newcommand{\dfn}[1]{\textsf{#1}\index{#1}}

\theoremstyle{plain} 
\newtheorem{thm}{Theorem}[section]

\newtheorem{lem}[thm]{Lemma}
\newtheorem{prop}[thm]{Proposition}
\newtheorem{cor}[thm]{Corollary}
\newtheorem{conj}[thm]{Conjecture}
\newtheorem*{conj*}{Conjecture}

\newtheorem{quest}[thm]{Question}

\theoremstyle{definition}
\newtheorem{defn}[thm]{Definition}
\newtheorem{example}[thm]{Example}  

\newtheorem*{def:gradedSS}{Definition \ref{def:gradedSS}}

\theoremstyle{remark}  
\newtheorem{rem}[thm]{Remark}

\newtheorem*{observ*}{Observation}

\newtheorem*{claim*}{Claim}

\newtheoremstyle{editorialNotes}
	{10mm} 
	{10mm}
	{\slshape}
	{-80pt}
	{\bfseries}
	{}
	{10mm}
	{}

\theoremstyle{editorialNotes}

\author[B. Stone]{Branden Stone}
\thanks{The author was partially funded by the NSF grant, Kansas Partnership for Graduate Fellows in K-12 Education (DGE-0742523).}

\title{Super-Stretched and Graded Countable Cohen-Macaulay Type}
\address{Branden Stone, Mathematics Program, Bard College/Bard Prison Initiative, P.O. Box 5000, Annandale-on-Hudson, NY 12504}
\email{bstone@bard.edu}

\begin{document}

\maketitle 
	
\begin{abstract}
	We define what it means for a \CM ring to to be \SuS and show that \CM rings of graded countable \CM type are \SuSp.  We use this result to show that rings of graded countable \CM type, and positive dimension, have possible $h$-vectors $(1)$, $(1,n)$, or $(1,n,1)$. Further, one dimensional standard graded Gorenstein rings of graded countable type are shown to be hypersurfaces; this result is not known in higher dimensions.  In the non-Gorenstein case, rings of graded countable \CM type of dimension larger than 2 are shown to be of minimal multiplicity.
\end{abstract}


\section{Introduction}\label{sec:prelim-def}

	A ring $R$ is said to be \dfn{standard graded} if, as an abelian group, it has a decomposition $R = \bigoplus_{i \gs 0} R_i$ such that $R_iR_j \subseteq R_{i+j}$ for all $i,j\gs 0$, $R = R_0[R_1]$, and $R_0$ is a field.  Further, we will always assume that a standard graded ring is Noetherian. Unless otherwise stated, we will denote by $(R,\m,k)$ the standard graded ring with $\m$ being the irrelevant maximal ideal, that is, $\m = \sum_{i =1}^\infty R_i$, and $k := R_0 = R/\m$ being an uncountable field.  A standard graded \CM ring $(R,\m,k)$ has \dfn{graded finite \CM type} (respectively, \dfn{graded countable \CM type}) if it has only finitely (respectively, countably) many indecomposable, maximal \CM modules up to a shift in degree. 
	
	The study of rings with finite and countable \CM type were studied extensively by M. Auslander and I. Reiten \cite{auslander86, auslander89, auslander89b} and an early survey paper was given by F.-O. Schreyer \cite{schreyer87}. Since then, more extensive surveys  have been published \cite{yoshino90, leuschke} detailing many interesting tools and results.  In a sequence of two papers, H. Kn\"orrer, R.-O. Buchweitz, G.-M. Greuel, and F.-O Schreyer \cite{knorrer87, buchweitz87} showed that if $R$ is a complete hypersurface containing an algebraically closed field $k$ (of characteristic different from 2), then $R$ is of finite \CM type if and only if $R$ is the local ring of a simple hypersurface singularity in the sense of \cite{arnold74}.  For example, if we let $k = \C$, then $R$ is one of the complete ADE singularities over $\C$.  That is, $R$ is isomorphic to $k\llbracket x, y, z_2,\ldots, z_d\rrbracket/( f )$, where $f$ is one of the following polynomials:
	\begin{align*}
	(A_n): &\  x^{n+1}+y^2+z_2^2+\cdots+z_d^2 ,\  n \gs 1; \\
	(D_n): &\  x^{n-1}+xy^2+z_2^2+\cdots+z_d^2 , \ n \gs 4; \\
	(E_6): &\  x^4+y^3+z_2^2+\cdots+z_d^2; \\
	(E_7): &\  x^3y+y^3+z_2^2+\cdots+z_d^2; \\
	(E_8): &\  x^5+y^3+z_2^2+\cdots+z_d^2.
	\end{align*}
In the countable case, it was further shown \cite{buchweitz87} that a complete hypersurface singularity over an algebraically closed uncountable field $k$ (with characteristic different from 2) has (infinite) countable \CM type if and only if it is isomorphic to one of the following
	\begin{align}
		(A_\infty): &\ k\llbracket x, y, z_2,\ldots, z_d\rrbracket /( y^2+z_2^2+\cdots+z_d^2 ); \label{eq1}\\ 
		(D_\infty): &\ k\llbracket x, y, z_2,\ldots, z_d\rrbracket /( xy^2+z_2^2+\cdots+z_d^2 ). \label{eq2}
	\end{align}	
Thus a complete characterization of hypersurfaces with countable \CM type is given (here we allow countable type to include finite type). 

	It has long been known that Gorenstein rings of finite \CM type are hypersurfaces (for a proof see \cite{herzog78}).  Combining this knowledge with the results above allows for a further characterization of Gorenstein rings of finite \CM type.  Given this outcome, there is a natural folklore conjecture: 
\begin{conj}\label{conj:gor}
	A Gorenstein ring of countable \CM representation type is a hypersurface.
\end{conj}
\noindent In Corollary \ref{cor:1n1gor}, we show that Conjecture \ref{conj:gor} is true in dimension one for a standard graded ring. 

The finite counterpart of Conjecture \ref{conj:gor} was of crucial importance in the classification of rings of graded finite \CM type.  In \cite{eisenbud88}, D. Eisenbud and J. Herzog completely classified the standard graded \CM rings of graded finite representation type in the category of graded maximal \CM modules over a ring $R$ and degree-preserving homomorphisms.  In doing this, they show that such rings are stretched as introduced by J. Sally \cite{sally79} (see Definition \ref{def:graded-stretched}).  
	
	 In this manuscript, the stronger notion of \SuS (see Definition \ref{def:gradedSS}) is developed and we are able to extend D. Eisenbud and J. Herzog's result to rings of graded countable type.  In particular, we show that standard graded \CM rings of graded countable \CM type are \SuSp.  This in turn forces the $h$-vectors to be of the form $(1)$, $(1,n)$, or $(1,n,1)$ for some $n$.  Similar observations were recently shown to be true in the context of wild representation type \cite{tovpyha12}.

	For a $d$-dimensional standard graded \CM ring $(R,\m,k)$, let $e(R)$ be the multiplicity of $\m$.  If $e(R)= \dim_k(\m/\m^2) - \dim R + 1 $ then $R$ is said to have \dfn{minimal multiplicity}.  In the case when $k$ is infinite, the following are equivalent:
		\begin{itemize}\label{min-mult-equiv}
			\item $R$ has minimal multiplicity;
			\item there exists a homogeneous regular sequence $x_1,\ldots,x_d$ such that the degree of $x_i$ is one and $\m^2 = (x_1,\ldots,x_d)\m$; 
			\item the $h$-vector of $R$ is of the form $(1,n)$.
		\end{itemize}
An interesting byproduct of the classification of D. Eisenbud and J. Herzog is that standard graded rings of dimension $d \gs 2$ with graded finite \CM type have minimal multiplicity. 
The authors ask if this result carries over to complete local rings in general.  While this question remains open, in Proposition \ref{prop:dim3-min-mult}, we are able to show that non-Gorenstein standard graded rings of dimension $d \gs 3$, with graded countable \CM type, have minimal multiplicity as well.

\subsection{Preliminaries}

Recall that the \dfn{associated graded ring}, $\gr{I}{R}$, is defined by
		\[
			\gr{I}{R} := R/I\ \oplus\ I/I^2\ \oplus\ I^2/I^3 \oplus\ \cdots.
		\]
For a standard graded ring $(R,\m,k)$, we have $R \simeq \gr \m R$. The \dfn{Hilbert function}, $h_R(n)$, is the vector space dimension of the $n^{th}$ summand of $\gr{\m}{R}$; that is, 
		\[
			h_R(n) := \dim_k (\m^n/\m^{n+1}).
		\]
	If $(R,\m,k)$ is a $d$-dimensional standard graded ring with infinite residue field $k$, then D. Northcott and D. Rees \cite{northcott54} showed there exists $\lst x d$ in $\m$ such that $\m^{n+1} = (\lst x d)\m^n$.  The sequence $\lst x d$ is called a \dfn{minimal reduction} of $\m$.  As $R$ is a standard graded ring, the minimal reduction of $\m$ consists of homogeneous elements of degree 1. 
A fact that will be used constantly is that in a \CM ring, minimal reductions are regular sequences.  Factoring out such a reduction yields an Artinian ring $\ol R$, and thus there exists an $s$ such that for $n > s$, the Hilbert function $h_{\ol R}(n) = 0$ and $h_{\ol R}(s) \not= 0$.  The sequence of integers 
	\[
		(h_{\ol R}(0), h_{\ol R}(1), \ldots, h_{\ol R}(s))
	\]
is called the \dfn{$h$-vector}\label{def:hvector} of $R$ and denote it by $h(R)$.  In general, the Hilbert function defines the \dfn{Hilbert series} $H_R (t) := \sum h_R (i) t^i$. This series can be represented as a rational function 
	\[
		H_R (t) = \frac{f(t)}{(1-t)^{d}},
	\]
where $f(t)\in \Z [t]$ and $f(1) \neq 0$ whenever $d >0$.  

 	The following definition of stretched first was given in 1979 by J. Sally \cite{sally79}.

\begin{defn}\label{def:graded-stretched}
	Let $(R,\m,k)$ be a standard graded \CM ring of dimension $d$.  We say $R$ is \dfn{stretched} if there exists a homogeneous minimal reduction $\x = (\lst x d)$ of $\m$ such that 
	\[
		\dim_k \left( \frac{R}{\x} \right)_i \ls 1
	\]
for all $i \gs 2$.
\end{defn}

In particular, if a ring $R$ is stretched, then there exists a minimal reduction $(\lst x d)$ of the maximal ideal $\m$ such that the $h$-vector is $(1,a,1,1,\ldots,1)$.  Here, 
\[
	a = \dim_k\left(\frac{\m}{\m^2 +(\lst x d)}\right) = \dim_k\left(\frac{\m}{\m^2}\right) - d.
\]

We extend the definition above by considering the Hilbert series of \emph{any} homogeneous system of parameters.

	\begin{defn}\label{def:gradedSS}
		A standard graded \CM ring $(R,\m,k)$ of dimension $d$ is said to be \dfn{\SuS} if for all homogeneous systems of parameters $\lst{x}{d}$, 
		\begin{equation}\label{eqn:gradedSS}
			\dim_k\left(\frac{R}{(\lst{x}{d})}\right)_i \ls 1
		\end{equation}
	for all $i\gs \sum \deg(x_j) - d +2$.
	\end{defn}

	If a ring $R$ is \SuSp, then for any homogenous system of parameters $(\lst x d)$, the $h$-vector is $(1,a_1,a_2,\ldots, a_{D-1},1,\ldots,1)$.  Here, $D = \sum \deg(x_i) - d +2$ and
		\[
			a_j = \dim_k\left(\frac{\m^j+(\lst x d)}{\m^{j+1} +(\lst x d)}\right).
		\]	
\begin{rem}\label{rem:SS-is-streched}
	As it is, if a standard graded ring is \SuS with an infinite residue field, then it is also stretched. To see this, choose a homogeneous minimal reduction $\lst x d$ of degree one.  Then Equation \eqref{eqn:gradedSS} holds for all $i \gs \sum \deg(x_j) - d + 2 = 2$; i.e. $R$ is stretched. 
\end{rem}

		As we see in the next example, if a ring is stretched, it is not necessarily \SuSp.

	\begin{example}\label{ex:stretched-ss}
		The standard graded \CM ring $R = \C[x,y]/(x^3y-xy^3)$ is stretched but not \SuSp.  To see this, notice that $x+2y$ is a regular element.  As a vector space over $\C$,
		\[
			\dim_\C\left(\frac{R}{(x+2y)}\right)_2 = \  \dim_\C\left(\frac{R}{(x+2y)}\right)_3 = 1 
		\]
	and
	\[
		\dim_\C\left(\frac{R}{(x+2y)}\right)_i = 0
	\]
	for all $i\gs 4$.  In order for $R$ to be \SuSp, $\dim_\C R_i \ls 1$ for all 
	\[
		i \gs \deg( (x+2y)^2 ) -1 +2 = 3.
	\]
	However, going modulo $(x+2y)^2$ yields
			\[
				\dim_\C\left(\frac{R}{(x+2y)^2}\right)_3 = 2.
			\]
	and thus $R$ is not \SuSp.
	\end{example}

	\section{Graded Nomenclature}\label{ch2}\label{section:graded-case}

	For a graded module $M$ over a graded ring $R$, we denote the \dfn{shift of $M$ by $n$} by $M(n)_i = M_{i+n}$.  Further, we say two $R$-modules $M$ and $N$ are \dfn{isomorphic up to shift in degree} if there exists an integer $n$, and a homogeneous isomorphism of degree zero, such that $M \simeq N(n)$.

	All of the following results will be considered in the category of finitely generated graded $R$-modules. In particular, we want to work in a subcategory consisting of graded maximal \CM modules.  The objects of the desired subcategory are the obvious choices, graded maximal \CM modules.  However, there is a bit of ambiguity as to what the maps ought to be; should we consider them to be homogeneous of degree zero, or will any graded map suffice? We illustrate this point in Example \ref{ex:EHrank}. Throughout this section, we give a precise definition of what subcategory we will be working in.  We also explore alternate ways of defining rings of graded finite (and countable) \CM type.

	\begin{example}\label{ex:EHrank}
		Let $(R,\m,k)$ be a standard graded \CM ring.  Note that $R$ is a graded indecomposable maximal \CM module and consider the family of graded maximal \CM modules $\Lambda = \{R \oplus R(-i)\}_{i\in \Z_{>0}}$.  There does not exist a graded isomorphism of any degree between two distinct modules in $\Lambda$, and each module is of rank 2.  Thus we have an infinite family of non-isomorphic (with respect to the grading), maximal \CM modules of bounded rank.  If the maps in the category are defined to be graded of arbitrary degree, then $\Lambda$ is composed of one indecomposable maximal \CM module up to graded isomorphism. 
		On the other hand, if we only consider homogeneous degree zero maps, $\Lambda$ is then composed of infinitely many, non-isomorphic indecomposable maximal \CM modules.  However, up to isomorphism and shifts of degree, there is only one indecomposable maximal \CM module.  
	\end{example}

		Depending on our ring, Corollary \ref{cor:finite-type} and Proposition \ref{prop:countable-type} allow us to handle the nuances arising in Example \ref{ex:EHrank}.

		For a general Noetherian ring $R$, let \dfn{$\Mod R$} denote the category of finitely generated $R$-modules.  Here the objects are defined as finitely generated $R$-modules and the morphisms are $R$-module homomorphisms. The subcategory \dfn{$\MCM R$} is the category of maximal \CM modules whose morphisms are defined by $R$-module homomorphisms between maximal \CM modules.  If $R$ is a standard graded ring, we define a subcategory of $\Mod R$ that respects the grading.  In particular, we let \dfn{$\Modgr R$} be the category whose objects are finitely generated graded modules.  The morphisms of $\Modgr R$ are graded $R$-module homomorphisms of degree zero.  As with $\MCM R$, we define the subcategory of graded maximal \CM modules by \dfn{$\MCMgr R$} where the morphisms are graded degree zero $R$-module homomorphisms. 


Using our new notation, we define what is meant by a standard graded ring of graded \CM type.

	\begin{defn}
		A standard graded \CM ring $(R,\m,k)$ ring is said to have \dfn{graded finite \CM type} (respectively, \dfn{graded countable \CM type}) if it has only finitely (respectively, countably) many indecomposable modules in $\MCMgr R$ up to a shift in degree.
	\end{defn}

	\subsection{Graded Finite \CM Type}	It is worth pointing out that there are a few possible choices in the definition of graded finite \CM type of a standard graded ring $R$.  For example, one could use any of the following characterizations for graded finite \CM type:
		\begin{enumerate}
			\item[(A)] there are finitely many graded indecomposable modules up to isomorphism in $\MCMgr R$;
			\item[(B)] there are finitely many graded indecomposable modules up to isomorphism and shifts in degree in $\MCMgr R$;
			\item[(C)] there are finitely many graded indecomposable modules up to isomorphism in $\MCM R$;
			\item[(D)] there are finitely many indecomposable modules up to isomorphism in $\MCM{\wh R}$.
		\end{enumerate}
	Here and throughout the rest of this work, we denote the completion with respect to the $\m$-adic topology by $\wh *$.  As it turns out, using (A) as the definition would not be very helpful, since in general $\{R(n)\}$ is an infinite family of non-isomorphic graded indecomposable maximal \CM modules.  In short, only the zero ring would have graded finite \CM type.  Thus we can safely remove (A) from the list of possible definitions.  Since we have adopted (B) as the definition, the question is, how do (C) and (D) fit into the picture?  In Corollary \ref{cor:finite-type}, we see that (B), (C) and (D) are equivalent definitions, which follows as consequence of the work of M. Auslander and I. Reiten.  

	\begin{prop}[{\cite[Proposition 8 and 9]{auslander89}}]\label{prop:auslanderprop89}
		Let $A, B$ be objects in $\MCMgr R$ where $(R,\m,k)$ is a standard graded \CM ring. 
		\begin{enumerate}
			\item The graded module $A$ is indecomposable in $\MCMgr R$ if and only if $\wh A$ is indecomposable in $\MCM{\wh R}$;

			\item If $A$ and $B$ are indecomposable, then $\wh A \simeq \wh B$  in $\MCM{\wh R}$ if and only if there is some integer $n$ such that $A \simeq B(n)$ in $\MCMgr R$.  
		\end{enumerate}
	\end{prop}

	\begin{cor}\label{cor:newGFCMT}
		Let $(R,\m,k)$ be a standard graded \CM ring and $M,N$ be indecomposable objects in $\MCMgr R$. Then, $M \simeq N$ in $\MCM R$ if an only if there is some integer $n$ such that $M \simeq N(n)$ in $\MCMgr R$.
	\end{cor}
	\begin{proof}
		This follows from the fact that completion is faithfully flat and Proposition \ref{prop:auslanderprop89} part (2).
	\end{proof}

		Another immediate corollary of Proposition \ref{prop:auslanderprop89} is the fact that the \CM type of the completion ``bounds'' the graded \CM type.
	\begin{cor}\label{cor:completion-graded}
		Let $(R,\m,k)$ be a standard graded \CM ring and $\wh R$ the $\m$-adic completion.  If $\wh R$ is of finite (respectively countable) \CM type, then $R$ is of graded finite (respectively graded countable) \CM type.
	\end{cor}

	The next proposition shows the equivalence of (B) and (C) for rings of finite \CM type.

	\begin{prop}\label{prop:finite-type}
		Let $(R,\m,k)$ be a standard graded \CM ring.  Then (B) and (C) are equivalent statements. In particular, either statement could be used as the definition of graded finite \CM type.
	\end{prop}
	\begin{proof}
		To see that (B) implies (C), notice that condition (C) has more isomorphisms in each class of indecomposable maximal \CM modules than there are in each class satisfying condition (B).  Thus, if (B) holds true, then (C) must be also be fulfilled.

		It is left to show that (C) implies (B).  By contradiction, assume there are infinitely many graded indecomposable modules in $\MCMgr R$ up to shifts in degree.  We let $\{M_\alpha\}_{\alpha \in \Lambda}$ be a family of representatives, one from each isomorphism class.  Let $\alpha,\beta \in \Lambda$ and assume that $M_\alpha \simeq M_\beta$ in $\MCM R$.  By Corollary \ref{cor:newGFCMT}, there exists an $n$ such that $M_\alpha \simeq M_\beta(n)$ in $\MCMgr R$.  In other words, $M_\alpha$ and $M_\beta$ are in the same isomorphism class up to shift.  Therefore we must have that $\alpha = \beta$. 
	\end{proof}

	When considering indecomposable maximal \CM $R$-modules $M,N$ in $\MCMgr R$, if there is an isomorphism between $M$ and $N$, then Corollary \ref{cor:newGFCMT} says there exists a graded isomorphism between the two modules.  Another nice result is that the ``finiteness'' of $\MCMgr R$ and $\MCM{\wh R}$ are the same.  

	\begin{thm}[{\cite[Theorem 5]{auslander89}}]\label{thm:fcmt-complete}
		Let $(R,\m,k)$ be a standard graded \CM ring and $\wh R$ the completion with respect to the maximal ideal $\m$. Then $R$ is of graded finite \CM type if and only if $\wh R$ is of finite \CM type.
	\end{thm}

	With Theorem \ref{thm:fcmt-complete} in hand, we are able to combine it with Proposition \ref{prop:finite-type} to obtain the following immediate corollary.

	\begin{cor}\label{cor:finite-type}
		Let $(R,\m,k)$ be a standard graded \CM ring and $\wh R$ the completion with respect to the maximal ideal $\m$.  Then conditions (B), (C), and (D) are equivalent.  
	\end{cor}

	\subsection{Graded Countable \CM Type}	As with graded finite \CM type, there are a few possible choices for the definition of graded countable \CM type of a standard graded ring $R$.  As described above, we could use any of the following for the definition:
		\begin{enumerate}\label{def:gct-choices}
			\item[(A')] there are countably many graded indecomposable modules up to isomorphism in $\MCMgr R$;
			\item[(B')] there are countably many graded indecomposable modules up to isomorphism and shifts in degree in $\MCMgr R$;
			\item[(C')] there are countably many graded indecomposable modules up to isomorphism in $\MCM R$;
			\item[(D')] there are countably many indecomposable modules up to isomorphism in the category $\MCM{\wh R}$.		
		\end{enumerate}
	Unlike the finite case, using (A') as the definition has potential. Notice that condition (A') is just removing the shifts and only allowing degree zero homomorphism between the modules.  By removing the shifts, we are only adding up to countably many new isomorphism classes with condition (A'). Hence (A') does not really add anything new.  In Proposition \ref{prop:countable-type}, we see that conditions (A'), (B'), and (C') are equivalent.  Further, Corollary \ref{cor:countable-type} describes the relation of (D') with the other statements.  

	\begin{prop}\label{prop:countable-type}
		Let $(R,\m,k)$ be a standard graded \CM ring.  Then (A'), (B'), and (C') are equivalent statements. In particular, any of the statements could be used as the definition of graded countable \CM type.
	\end{prop}
	\begin{proof}
		To show that (A') implies (B'), notice that by removing the shifts we are adding more isomorphism classes.  Therefore (B') follows.  A similar argument as in Proposition \ref{prop:finite-type} shows that (B') implies (C').

		To see that (B') implies (A'), assume by contradiction that there exists uncountably many graded indecomposable maximal \CM modules up to isomorphism in $\MCMgr R$. Let $\{ M_\alpha\}_{\alpha\in \Lambda}$ be a family of representatives, one from each indecomposable class.  Consider the isomorphism classes up to shifts in degrees.  That is for each $\alpha \in \Lambda$, there exists a subset $I \subseteq \Lambda$, with the property that for each $\beta\in I$, there exists an integer $n$ such that $M_\alpha \simeq M_\beta(n)$.  Let $\beta,\gamma \in I$ and assume that there exist an integer $n$ such that
		\[
			M_\gamma(n) \simeq M_\alpha \simeq M_\beta(n).
		\]
	As all of the isomorphisms above are degree zero, we have that $M_\beta = M_\gamma$ (i.e. $\beta = \gamma$).  Hence, when we include the shifts to our assumption, for each $\alpha \in \Lambda$ we only associate countably many indecomposables up to shifts.  Hence there are uncountably many graded indecomposable modules in $\MCMgr R$ that are not isomorphic up to shifts in degrees, a contradiction.

		It is left to show that (C') implies (B').  By contradiction, assume there are uncountably many graded indecomposable modules in $\MCMgr R$ up to shifts in degree.  Let $\{M_\alpha\}_{\alpha \in \Lambda}$ be an uncountable family of representatives from each isomorphism class.  We wish to form the isomorphism classes described in (C').  Let $\alpha,\beta \in \Lambda$ and assume that $M_\alpha \simeq M_\beta$ in $\MCM R$.  Thus by Corollary \ref{cor:newGFCMT}, there exists an $n$ such that $M_\alpha \simeq M_\beta(n)$ in $\MCMgr R$.  In other words, $M_\alpha$ and $M_\beta$ are in the same isomorphism class up to shift.  Therefore we must have that $\alpha = \beta$.  
	\end{proof}

	\begin{cor}\label{cor:countable-type}
		Let $(R,\m,k)$ be a standard graded \CM ring.  If condition (D') holds, then so do the statements (A'), (B'), and (C').
	\end{cor}
	\begin{proof}
		This is a direct application of Corollary \ref{cor:completion-graded} and Proposition \ref{prop:countable-type}.
	\end{proof}

		In order to obtain the much desired converse to Corollary \ref{cor:countable-type}, we need a countable version of M. Auslander and I. Reiten's result in Theorem \ref{thm:fcmt-complete}. The proof of Theorem~\ref{thm:fcmt-complete} relies on the fact that rings of graded finite \CM type have an isolated singularity \cite[Proposition 4]{auslander89}.
As such, it is worth reconsidering the relation of condition (D') and conditions (A'), (B') and (C'), with the added assumption that the ring has an isolated singularity. With the extra hypothesis, it might be possible to extend Theorem~\ref{thm:fcmt-complete} to the countable case, and thus obtain a partial converse to Corollary~\ref{cor:countable-type}. We leave this as a question.

	\begin{quest}\label{thm:ccmt-complete}
		Let $(R,\m,k)$ be a standard graded \CM ring and $\wh R$ the completion with respect to the maximal ideal $\m$. If $R$ has an isolated singularity and is of countable graded \CM type, then is $\wh R$ of countable \CM type?
	\end{quest}

	\section{Super-Stretched Standard Graded Cohen-Macaulay Rings}

	The goal of this section is to show equivalent characterizations of \SuS (Theorem \ref{thm:SSequiv}). To do this, we need to build up the theory of \SuS standard graded rings with a few useful propositions. For convenience we recall the following definition.
	
	\begin{def:gradedSS}
		A standard graded \CM ring $(R,\m,k)$ of dimension $d$ is said to be \dfn{\SuS} if for all homogeneous systems of parameters $\lst{x}{d}$, 
		\begin{equation*}
			\dim_k\left(\frac{R}{(\lst{x}{d})}\right)_i \ls 1 \tag{\ref{eqn:gradedSS}}
		\end{equation*}
	for all $i\gs \sum \deg(x_j) - d +2$.
	\end{def:gradedSS}
	
 	Throughout this section, we will use the following notation.  Let $\lst y n$ and $\lst x m$ be sequences in a ring $R$ such that $(\lst y n) \subseteq (\lst x m)$. Let $A = (a_{ij})$ be the $n \times m$ matrix of elements in $R$ such that $y_i = \sum_{j=1}^m a_{ij}x_j$.  Set $\Delta = \det(A)$ and denote the containment of $(\lst y n)$ in $(\lst x m)$ by 
\[
	(\lst y n) \overset{A}{\subseteq} (\lst x m).
\]
Further, we define $(\lst y n)^{[t]} := (\lst{y^t} n)$. With this notation, we recall that for a standard graded ring $(R,\m,k)$ and  a homogeneous regular sequence $\lst x n$,
		\begin{equation}\label{lem:monConj}
			(\lst x n)^{[t]}:(x_1\cdots x_n)^{t-1} = (\lst x n).
		\end{equation}
for all $t \gs 1$.

	\begin{prop}\label{lem:sopMap}
		Let $(R,\m,k)$ be a standard graded \CM ring of dimension $d$.  If $(\lst y d ) \overset{A}{\subseteq} (\lst x d )$ are ideals, each generated by a homogeneous system of parameters, then the map $R/(\lst x d) \stackrel{\cdot\ \Delta\ }{\longrightarrow} R/(\lst y d)$ is injective.
	\end{prop}
	\begin{proof}
		Notice that the above map is well-defined as $\Delta (\lst x d) \subset (\lst y d)$.  Let $r \in R$ such that $r \cdot \Delta \in (\lst y d)$. Since $\lst y d$ is a homogeneous system of parameters, there exists a positive integer $t$ and a matrix $B$ such that $(\lst x d)^{[t]} \overset B \subseteq (\lst y d)$. Hence we have the following inclusions:
	\begin{align*}
		&(\lst{x} d)^{[t]} \overset{B}{\subseteq} (\lst{y}d) \overset{A}{\subseteq} (\lst{x}d)\\
		&(\lst{x}d)^{[t]} \overset{D}{\subseteq} (\lst{x}d),
	\end{align*}
	where $D$ is the diagonal matrix with entries $x_i^{t-1}$.	Let $E=AB$. By \cite[Corollary 2.5]{fouli2011} we obtain 
	\[
	(x_1 \cdots x_d)^{td}(\det\ E -\det\ D) \in (\lst{x}d)^{[td+t]}.
	\]
	As $\det\ D = (x_1\cdots x_d)^{t-1}$, we have
	\[
	(x_1 \cdots x_d)^{td} (\det\ B) \Delta - (x_1 \cdots x_d)^{td}(x_1 \cdots x_d)^{t-1} \in (\lst{x}d)^{[td+t]}
	\]
	and thus multiplication by $r$ yields
	\begin{equation}\label{eq:det3}
	r(x_1 \cdots x_d)^{td}(\det\ B) \Delta-r(x_1 \cdots x_d)^{td}(x_1 \cdots x_d)^{t-1} \in (\lst{x}d)^{[td+t]}.
	\end{equation}
	Since $r\Delta \in (\lst y d)$, we have that
	\begin{equation}\label{eq:det1}
		r\cdot(x_1 \cdots x_d)^{td}\cdot\det B\cdot\Delta \in (x_1 \cdots x_d)^{td}\cdot\det B\cdot (\lst{y}d).
	\end{equation}
	By definition of $B$, $(\lst y d)\cdot \det\ B \subseteq (\lst x d)^{[t]}$ and hence
	\begin{equation}\label{eq:det2}
		(x_1 \cdots x_d)^{td}\cdot\det B\cdot (\lst{y}d) \subseteq (x_1 \cdots x_d)^{td}(\lst{x}d)^{[t]} \subset (\lst{x}d)^{[td+t]}.
	\end{equation}
	Combining (\ref{eq:det1}) and (\ref{eq:det2}) we see that $r\cdot(x_1 \cdots x_d)^{td}\cdot\det B\cdot\Delta \in (\lst{x}d)^{[td+t]}$.  Therefore, (\ref{eq:det3}) gives $r(x_1 \cdots x_d)^{td}(x_1 \cdots x_d)^{t-1} \in (\ul{x})^{[td+t]}$ and thus 
	\[
		r(x_1 \cdots x_d)^{td+t-1} \in (\ul{x})^{[td+t]}. 
	\]
	Since $R$ is \CMp, we have that our sequence $(\lst x d)$ is actually a homogeneous regular sequence.  Applying \eqref{lem:monConj} shows that $r \in (\lst x d)$ and hence multiplication by $\Delta$ is injective.
	\end{proof}

	\begin{prop}\label{prop:CMlinkSS}
		Let $(R,\m,k)$ be a standard graded \CM ring of dimension $d$.  If $\lst{y}{d}$ is a homogeneous system of parameters satisfying \eqref{eqn:gradedSS}, and $\lst{x}{d}$ is a homogeneous system of parameters such that $(\lst{y}{d}) \overset{A}{\subseteq} (\lst{x}{d})$, then $\lst{x}{d}$ satisfies \eqref{eqn:gradedSS} as well.
	\end{prop}
	\begin{proof}
		Let $\deg(y_i) = f_i$ and $\deg(x_i) = e_i$.  This forces $\deg(a_{ij}) = f_i-e_j$ and hence $\deg(\Delta)= \sum_{i=1}^d f_i - \sum_{i=1}^d e_i$.  Since $R$ is \CMp, note that $\lst{y}{d}$ and $\lst{x}{d}$ are regular sequences.  Let $c = \sum_{i=1}^d e_i -d +2$.  By Proposition \ref{lem:sopMap}, multiplication by $\Delta$ is an injection. Hence we have the following linear map of vector spaces
		\[
			\xymatrixrowsep{5mm}
			\xymatrixcolsep{8mm}
			\xymatrix
				{
					\left(\frac{R}{(\lst{x}{d})}\right)_c \ar@{^{(}->}[r]^-{\Delta} & \left(\frac{R}{(\lst{y}{d})}\right)_{c+\deg(\Delta)}.
				}
		\]
Combining the above map with the fact that $\lst{y}{d}$ satisfies \eqref{eqn:gradedSS}, we see that
	\[
		\dim_k(R/({\lst{x}{d}}))_c \ls \dim_k(R/(\lst{y}{d}))_{c+\deg(\Delta)} \ls 1,
	\]
which is the desired result.
\end{proof}

	This next proposition distinguishes \SuS rings from stretched rings.  
	
	\begin{prop}\label{prop:gradded-ssReduction}
		If $(R,\m,k)$ is a standard graded \CM ring of dimension $d>0$ that is \SuSp, then for all homogeneous minimal reductions $(x_1,\ldots,x_d)$ of the maximal ideal $\m$, we have $\m^3 = (x_1,\ldots,x_d)\m^2$.
	\end{prop}
	\begin{proof}
		Induct on $d$.  For the dimension one case let $x$ be a minimal reduction. (Note that $\deg(x)=1$.)  Because $R$ is super stretched, we have that $\dim_k (R/xR)_2 = 1$. Further, 
		\[
			\dim_k \frac{\m^2+(x)}{\m^3+(x)} = \dim_k \frac{\m^2}{\m^3+((x)\cap\m^2)} = \dim_k \frac{\m^2}{\m^3+x\m} = 1.
		\]
	Note that the second equality is always true as $(x)\cap \m^2 = x\m$ if $x\notin \m^2$. The displayed equality says that there is a $y\in \m^2 - (x\m +\m^3)$ such that $\m^2 = x\m + (y) + \m^3$.  By Nakayama's lemma we have $\m^2 = x\m +(y)$.  

	We now consider $R$ modulo $x^2$.  Due to the grading we have that $(x^2)\cap \m^3=x^2\m$.  This gives us 
		\[
			\dim_k \frac{\m^3}{\m^4 + x^2\m} = 1.
		\]
	Thus, as before, there exists $z\in \m^3 - (x^2\m + \m^4)$ such that $\m^3 = x^2\m + (z) + \m^4$.  Nakayama's lemma shows that $\m^3 = x^2\m +(z)$.  

	Notice that we can choose $z$ to be anything in $\m^3 - (x^2\m + \m^4)$.  We would like to choose $z = xy$, but first we must show 
	\begin{claim*}
		The element $xy$ is not in $x^2\m + \m^4$.  
	\end{claim*}

	If the claim holds, then we have 
		\begin{align*}
			\m^3 & = x^2\m +(z)\\
				& = x^2\m +(xy) \\
				& = x(x\m +(y)) \\
				& = x\m^2.
		\end{align*}
	This is the desired result for dimension one.

	To show the claim, let $n$ be minimally chosen such that $\m^n = x\m^{n-1}$ and suppose $xy \in x^2\m + \m^4$. Since $\m^2 = x\m+(y)$, we have that $x\m^2 \subseteq x^2\m +\m^4$. Assume that $n>3$ and multiply by $\m^{n-3}$. As $R$ is \CMp, we have that $x$ is a non-zerodivisor.  Cancel the $x$'s to observe that $\m^{n-1} \subseteq x\m^{n-2} + \m^n$.  This forces $\m^{n-1}=x\m^{n-2}$, a contradiction since $n$ was chosen to be minimal. Thus $xy\notin (x^2\m+\m^4)$.

	For higher dimensions we may assume that 
	\[
		\frac{\m^3+(x_d)}{(x_d)} = \frac{(\lst{x}{d-1})\m^2+(x_d)}{(x_d)}.
	\]
	Lifting gives us the inclusion
	\[
		\m^3 \subseteq (\lst{x}{d-1})\m^2 + (x_d).
	\]
	Notice by the grading, and the regularity of $x_d$, we have that $\m^2 = (\m^3:x_d)$ and hence
	\begin{align*}
		\m^3 &= (\lst{x}{d-1})\m^2 + x_d(\m^3:x_d) \\
		& = (\lst{x}{d})\m^2.
	\end{align*}
\end{proof}

\begin{rem}
	When $R$ is a zero-dimensional \SuS graded ring, Proposition~\ref{prop:gradded-ssReduction} fails as can be seen by the ring $k[x]/(x^4)$.  
\end{rem}

	Before moving on to Proposition \ref{prop:stretchFrob}, we notice that for elements $\lst a k$ of a ring $R$, and for every positive integer $m$,
		\begin{equation}\label{lem:prodFrob}
			(\lst{a^m}{k})(\lst a k)^{(k-1)(m-1)} = (\lst a k)^{(m-1)k +1}.
		\end{equation}
	A proof of this can be found in \cite[Equation 8.1.6]{swanson}.

	\begin{prop}\label{prop:stretchFrob}
		Let $(R,\m,k)$ be a $d$-dimensional standard graded ring and let the ideal $(\lst{x}{d})$ be a homogeneous reduction of $\m$ such that $(\lst{x}{d})\m^2 = \m^3$. If $R$ is stretched, then $(x_1^t,\ldots,x_d^t)$ satisfies (\ref{eqn:gradedSS}) for all $t >0$.
	\end{prop}
	\begin{proof}
		We would like to show that for all $t >0$, 
		\[
			\dim_k \left(\frac{R}{(\lst{x^t}{d})}\right)_i \ls 1
		\]
	for each $i \gs dt - d +2$. In particular, we need that
	\[
		\dim_k \left( \frac{\m^{dt-d+2}}{\m^{dt-d+3}+ (\lst{x^t}{d})\cap \m^{dt-d+2}}\right) \ls 1.
	\]
	To show this, we need to first show the equality
	  	\begin{equation}\label{eq:dimequal}
			\frac{\m^{dt-d+2}}{\m^{dt-d+3}+ (\lst{x^t}{d})\cap \m^{dt-d+2}} = \frac{\m^{dt-d+2}}{ (\lst{x^t}{d})\m^{dt-d+2-t}}.
	  	\end{equation}
	In order to see equation \eqref{eq:dimequal}, notice that 
	\begin{equation}\label{eq:denom}
		\m^{dt-d+3}+ (\lst{x^t}{d})\cap \m^{dt-d+2} = \m^{dt-d+3}+ (\lst{x^t}{d})\m^{dt-d+2-t}.
	\end{equation}
	As $(\lst x d)\m^2 = \m^3$, we have that for any positive integer $N$, $(\lst x d)^N \m^2 = \m^{N+2}$.  Further, if we consider $\m^{dt-d+3}$ and let $N = d(t-1)+1$, we have that
	\begin{align*}
		\m^{dt-d+3} = \m^{d(t-1)+3} = (\lst x d)^{d(t-1)+1} \m^2.
	\end{align*}
	By \eqref{lem:prodFrob} we have that $(\lst x d)^{d(t-1)+1} = (\lst{x^t} d)(\lst x d)^{(d-1)(t-1)}$.  Therefore we have that
	\begin{align*}
		\m^{dt-d+3} &= (\lst{x^t} d)(\lst x d)^{(d-1)(t-1)} \m^2 \\
		 &= (\lst{x^t} d)\m^{(d-1)(t-1)+2} \\
		 & \subseteq (\lst{x^t} d)\m^{(d-1)(t-1)+1} = (\lst{x^t}{d})\m^{dt-d+2-t}.
	\end{align*}
	Applying this fact to Equation \eqref{eq:denom} allows us to write
	\[
		\m^{dt-d+3}+ (\lst{x^t}{d})\cap \m^{dt-d+2} = (\lst{x^t}{d})\m^{dt-d+2-t},
	\]
	and hence equality holds in Equation \eqref{eq:dimequal}.

		The next step is to show that 
	\begin{equation}\label{eq:powers}
		(\lst x d)^{dt-d} = (\lst{x^t} d)(\lst x d)^{dt-d-t} + (x_1x_2\cdots x_d)^{t-1}.
	\end{equation}
	Notice that the generators of $(\lst x d)^{dt-d-t}$ are all the monomials in $\lst x d$ of degree $dt-d-t$.  Thus the generators of the product $(\lst{x^t} d)(\lst x d)^{dt-d-t}$ are monomials $m$ in $\lst x d$ of degree $d(t-1)$ such that $x_j^t|m$ for some $j = 1, 2, \ldots d$.  Call the set of these monomials $M$. The monomials in $M$ are also a part of a minimal generating set of the ideal $(\lst x d)^{dt-d}$.  In fact, the set $N = \{ x_1^{n_1}x_2^{n_2}\cdots x_d^{n_d} \vl \sum n_i = dt-d \}$ is a generating set for the ideal $(\lst x d)^{dt-d}$.  Hence, the elements in $N$ that are not in $M$ are
	\[
		N\setminus M = \{ m = x_1^{n_1}x_2^{n_2}\cdots x_d^{n_d} \vl \sum n_i = d(t-1) \text{ and such that } x_j^t \not| m \text{ for all } j\}.
	\]
	This implies that $m \in N\setminus M$ is an element of the ideal $(x_1x_2 \cdots x_d)^{t-1}$ as $\deg(m) = d(t-1)$ and $x_j^t$ does not divide $m$.  We therefore have the equality in Equation \eqref{eq:powers}.

		Since $\lst x d$ is a reduction, we can write $\m^{dt-d+2} = \m^2 (\lst x d)^{dt-d}$. Combining this with Equation \eqref{eq:powers} yields
	\begin{align*}
		\m^{dt-d+2} &= \m^2 (\lst x d)^{dt-d}\\
		 	& = \m^2 ( (\lst{x^t} d)(\lst x d)^{dt-d-t} + (x_1x_2\cdots x_d)^{t-1} ) \\
			& = (\lst{x^t} d)(\lst x d)^{dt-d-t}\m^2 + (x_1x_2\cdots x_d)^{t-1}\m^2 \\
			& = (\lst{x^t} d)\m^{dt-d-t+2} + (x_1x_2\cdots x_d)^{t-1}\m^2.
	\end{align*}
	Because $R$ is stretched, we may proceed as in Proposition \ref{prop:gradded-ssReduction} and choose a $y\in \m^2 - ((\lst x d)\m +\m^3)$ such that $\m^2 = (\lst x d)\m +(y)$.  Substituting into the above relation yields
	\begin{align*}
		\m^{dt-d+2} & = (\lst{x^t} d)\m^{dt-d-t+2} + (x_1x_2\cdots x_d)^{t-1}((\lst x d)\m +(y)) \\
			&  = (\lst{x^t} d)\m^{dt-d-t+2} + (y(x_1x_2\cdots x_d)^{t-1}).
	\end{align*}
	Thus, modulo $(\lst{x^t} d)\m^{dt-d-t+2}$, we have that $\m^{dt-d+2}$ is one dimensional. Combining this fact with \eqref{eq:dimequal} give the desired result.
	\end{proof}

		We are now ready to state and prove some equivalent characterizations of \SuSp. This is also the main result of this section and is used to show that rings of graded countable \CM type are \SuS (see Theorem \ref{thm:countSS}).

	\begin{thm}\label{thm:SSequiv}
		Let $(R,\m,k)$ be a standard graded \CM ring of dimension $d>0$.  The following are equivalent:
		\begin{enumerate}
			\item\label{thm:SSequiv1} $R$ is \SuSp;
			\item\label{thm:SSequiv2} $R$ is stretched and $J\m^2 = \m^3$ for every homogeneous reduction $J$ of the maximal ideal;
			\item\label{thm:SSequiv3} $R$ is stretched and $J\m^2 = \m^3$ for some homogeneous reduction $J$ of the maximal ideal.
		\end{enumerate}
	\end{thm}
	\begin{proof}
		$(\ref{thm:SSequiv1}) \Rightarrow (\ref{thm:SSequiv2})$: This is an application of Proposition \ref{prop:gradded-ssReduction}.

		$(\ref{thm:SSequiv2}) \Rightarrow (\ref{thm:SSequiv3})$: This is straightforward.

		$(\ref{thm:SSequiv3}) \Rightarrow (\ref{thm:SSequiv1})$: Assume that $J$ is as in \eqref{thm:SSequiv3} and is generated by $\lst{x}{d}$.  Let $(\lst{y}{d})$ be an ideal of $R$ generated by a homogeneous system of parameters.  We have that there exists a $t$ such that $\m^t \subseteq (\lst{y}{d})$.  In particular, $(\lst{x^t}{d}) \subseteq (\lst{y}{d})$.  By proposition \ref{prop:CMlinkSS}, $(\lst{y}{d})$ satisfies (\ref{eqn:gradedSS}) since $(\lst{x^t}{d})$ satisfies (\ref{eqn:gradedSS}) by proposition \ref{prop:stretchFrob}.  Therefore, $R$ is \SuSp.
	\end{proof}

	\subsection{Super-Stretched and $h$-vectors} The next two results are immediate corollaries of Theorem \ref{thm:SSequiv} that describe the $h$-vector of a \SuS ring.

	\begin{cor}\label{cor:SShvector}
		Assume that $(R,\m,k)$ is a standard graded \SuS ring of dimension $d>0$ with infinite residue field $k$. Then the $h$-vector of $R$ is of one of the following forms: $(1)$, $(1,n)$, or $(1,n,1)$ for some non-zero positive integer $n$.  
	\end{cor}
	\begin{proof}
		Let $J$ be a minimal reduction of the maximal ideal $\m$.  Since $R$ is \SuSp, we have by Theorem \ref{thm:SSequiv} that $J\m^2 = \m^3$.  Let $\ol R = R/J$ and notice that we have
		\[
			h_{\ol R}(3) = \dim_k \left( \frac{\m^3}{\m^4 + J\cap \m^3} \right) = \dim_k \left( \frac{\m^3}{\m^4 + J\m^2} \right) = \dim_k \left( \frac{\m^3}{\m^4 + \m^3} \right) = 0.
		\]
	This forces $h_{\ol R}(n) = 0 $	 for all $n > 2$.  The fact that $R$ is stretched forces $h_{\ol R}(2) \ls 1$.  Therefore, the only possible $h$-vectors are $(1)$, $(1,n)$, or $(1,n,1)$ where $n$ is a non-zero positive integer.  
	\end{proof}

	\begin{cor}
		A standard graded hypersurface with positive dimension and multiplicity at most $3$ is \SuSp.
	\end{cor}
	\begin{proof}
		Let $R$ be a hypersurface with multiplicity $e \ls 3$.  As the sum of the $h$-vector is the multiplicity, the only possible $h$-vectors are $(1)$, $(1,1)$, and $(1,1,1)$.  All of these satisfy condition $\ref{thm:SSequiv3}$ of Theorem \ref{thm:SSequiv}.
	\end{proof}

	\section{Super-Stretched and Graded Countable Type}

		In this section we generalize \cite[Theorem A]{eisenbud88} to standard graded rings of graded countable \CM type.  The proof of Theorem \ref{thm:countSS} is an extension of D. Eisenbud and J. Herzog's proof and follows the same basic outline.	The following lemma is helpful in proving Theorem \ref{thm:countSS}.  

	\begin{lem}\label{lem:ideals}
		If $(R,\m,k)$ is a standard graded ring such that $\dim_k(R_i)>1$ for some $i > 0$, then there exists at least $|k|$ many distinct homogeneous principle ideals in $R$.
	\end{lem}
	\begin{proof}
		Let $x,y$ be distinct basis elements of $R_i$ and let $\alpha,\beta \in k$.  Assume that 
		\[
			(x+\alpha y)=(x+\beta y).
		\]
	Since $R$ is graded, there exists $\gamma \in k$ such that $x+\alpha y = \gamma(x+\beta y)$ in $R_i$. Hence
	\begin{equation}\label{eq:linear-dep}
		(1-\gamma)x + (\alpha - \beta\gamma) y = 0
	\end{equation}
	in $R$. In particular, this relation holds in the vector space $R_i$ as $\alpha,\beta,\gamma\in k$.  Thus the coefficients of $x$ and $y$ are zero and we have that $\gamma = 1$ and $\alpha = \beta$.  Therefore the conclusion follows.
	\end{proof}

	We are now ready to prove the generalization of D. Eisenbud and J. Herzog's result.  

	\begin{thm}\label{thm:countSS}
		Let $(R,\m,k)$ be a standard graded \CM ring of dimension $d>0$ with uncountable residue field $k$.  If $R$ is of graded countable \CM type then it is \SuSp.
	\end{thm}
	\begin{proof}
		Assume $R$ is not \SuS and let $x_1,\ldots,x_d$ be a homogeneous system of parameters such that
		\[
			\dim_k(R/(x_1,\ldots,x_d))_c \gs 2 
		\]
	for some 
		\begin{equation}\label{eqn:countSS1}
			c \gs \sum_{j=1}^d \deg(x_j) -d + 2.	
		\end{equation}
	Since $\dim_k(R/(x_1,\ldots,x_d))_c \gs 2$, then $\dim_k(R/(x_1,\ldots,x_d))_l \gs 2$ for all $2 \ls l \ls c$.  Therefore, without losing any generality, we can assume equality in \eqref{eqn:countSS1}.  Let $\ol R = R/(x_1,\ldots,x_d)$ and consider $\ol{y} \in (\ol{R})_c$.  Define $I_{\ol{y}} \subseteq R$ to be the preimage of $(\ol{y})$.  For each $\ol y \in \ol R_c$, we shall associate a graded maximal \CM module $M_{\ol y}$ such that the family $\{M_{\ol y}\}_{\ol y \in \ol R_c}$ has the following properties: 
	\begin{enumerate}
		\item\label{thm:countSS2} Let $(M_{\ol y})_i$ be the graded components of $M_{\ol y}$.  We have that $\dim_k(M_{\ol y})_{<t} = 0$ and $\dim_k(M_{\ol y})_t = 1$ where $t = \sum \deg(x_j)$.  
		\item\label{thm:countSS2.5} There is a unique indecomposable summand $N_{\ol y}$ of $M_{\ol y}$ such that $(N_{\ol y})_t = (M_{\ol y})_t$.
		\item\label{thm:countSS3} If $\ol{M_{\ol y}} = {M_{\ol{y}}/(\lst{x}{d})M_{\ol{y}}}$, then $\ann_R (\ol{M_{\ol y}})_t = I_{\ol{y}}$ where $t$ is as in \eqref{thm:countSS2}.  
	\end{enumerate}
	Assuming these three claims, we show there exist uncountably many graded indecomposable maximal \CM modules up to isomorphism in $\MCMgr R$.  Hence by Proposition \ref{prop:countable-type}, $R$ cannot be of graded countable \CM type.  

		As in \eqref{thm:countSS2.5}, let $N_{\ol y}$ and $N_{\ol y'}$ be the unique indecomposable summands of $M_{\ol y}$ and $M_{\ol y'}$ for $\ol{y},\ol{y}' \in \ol{R}_c$.   Suppose that there is an isomorphism $N_{\ol{y}} \simeq N_{\ol{y}'}$ in $\MCMgr R$. Thus we have that $Re \mapsto Re'$ where $e$ and $e'$ are generators of $(N_{\ol y})_t$ and $(N_{\ol y'})_t$ (respectively).  This implies
		\[
			N_{\ol{y}}/(x_1,\ldots,x_d)N_{\ol{y}} \simeq N_{\ol{y}'}/(x_1,\ldots,x_d)N_{\ol{y}'}.
		\]   
	Once again, under this isomorphism, $R\ol{e}\mapsto R\ol{e}'$. From \eqref{thm:countSS3}, we have that $\ann_R(\ol{e}) = I_{\ol{y}}$ and $\ann_R(\ol{e}') = I_{\ol{y}'}$, which forces $I_{\ol{y}} = I_{\ol{y}'}$.  Thus $(\ol{y}) = (\ol{y}')$.  Note that $\dim_k(\ol{R}_c) \gs 2$, so by Lemma \ref{lem:ideals} there exist uncountably many ideals $(\ol{y})$, where $\ol{y} \in \ol{R}_c$. As such there must be uncountably many graded indecomposable maximal \CM modules up to isomorphism in $\MCMgr R$ and we are finished by Proposition \ref{prop:countable-type}.

	To show property \eqref{thm:countSS2}, consider the Koszul complex $\mathcal K$ of the homogeneous system of parameters $x_1,\ldots,x_d$, 
	\[
		\xymatrixrowsep{5mm}
		\xymatrixcolsep{8mm}
		\xymatrix
			{
				\mathcal K:  0 \ar[r] & K_d \ar[r] & \cdots \ar[r] & K_1 \ar[r]  & R \ar[r] &  R/(x_1,\ldots,x_d)R  \ar[r] & 0.
			}
	\]
	Note that for $J\subseteq \{1,2,\ldots,d \}$, $ K_i \simeq \bigoplus_{|J| = i} R (-\sum_{j\in J} \deg{x_j})$.  Let $\Omega_i$ be the $i^{th}$ syzygy of $\mathcal K$ and fix $\ol y \in (\ol R)_c$.  Further, let $I_{\ol{y}} = (\lst x d , y)$ be the preimage of $(\ol{y})$ and consider the minimal resolution $\mathcal F$ of $R/I_{\ol y}$.  From $\mathcal F$, we have the short exact sequence
	\[
		\xymatrixrowsep{5mm}
		\xymatrixcolsep{8mm}
		\xymatrix
			{
				0 \ar[r] & M_2 \ar[r] & \bigoplus R(-\deg(x_i))\oplus R(-c) \ar[r] & (\lst{x}{d}, y) \ar[r] & 0
			}
	\]
	where $M_2$ is the second syzygy of $R/I_{\ol y}$. As such, consider the commutative diagram of degree zero maps

	\[
		\xymatrixrowsep{5mm}
		\xymatrixcolsep{8mm}
		\xymatrix
			{
				0 \ar[r] & \Omega_2 \ar[r] \ar@{^{(}->}[d] & K_1 \ar[r]\ar@{^{(}->}[d] & (\lst{x}{d})_{\ } \ar[r]\ar@{^{(}->}[d] & 0 \\
				0 \ar[r] & M_2 \ar[r] & \bigoplus R(-\deg(x_j))\oplus R(-c) \ar[r] & (\lst{x}{d}, y) \ar[r] & 0
			}
	\]
and define
	\(
		 t_i := \max_{|J| = i}\left\{ \sum_{j \in J} \deg(x_j) \right\}
	\)
for $1 \ls i \ls d$. Notice that 
    \[
        c + i = \sum_{j=1}^d \deg(x_j) -(d-(i+1)) + 1 > t_{i+1}.
    \]

Let $z \in M_2$ be a non-zero element such that $\deg(z) \ls t_2$ (this element exists as $\Omega_2 \hookrightarrow M_2$).  We claim that $z \in \im(\Omega_2 \hookrightarrow M_2)$. If not, then $\deg(z) \gs c+1$ as the degrees of elements in $M_2$ that depend on $y$ are bounded below by $c+1$.  However $c+1 > t_2$, a contradiction.  Hence we can write $M_2 = \im(\Omega_2) + \bigoplus_{i> t_2}(M_2)_i$.

    By induction on the homological degree, along with the fact that $c+i > t_{i+1}$, we have that $\Omega_{d-1} \hookrightarrow M_{d-1}$ and $M_{d-1} = \im(\Omega_{d-1}) + \bigoplus_{i>t_{d-1}} (M_{d-1})_i$.  This gives rise to the commutative diagram of degree zero maps,
    \[
		\xymatrixrowsep{5mm}
		\xymatrixcolsep{8mm}
		\xymatrix
			{
				0 \ar[r] & \Omega_d \ar[r] \ar@{^{(}->}[d] & K_{d-1} \ar[r]\ar@{^{(}->}[d] & \Omega_{d-1} \ar[r]\ar@{^{(}->}[d] & 0 \\
				0 \ar[r] & M_d \ar[r] & \ds \bigoplus_{|J| = d-1} R(-\sum_{j\in J} \deg(x_j))\oplus \bigoplus R(-a_{d-1,i}) \ar[r] & M_{d-1} \ar[r] & 0 .
			}
	\]
As above, we chose $z \in M_d$ such that $\deg(z) \ls t_d$.  Since $c+d > t_d$, we have that $z\in \im(\Omega_d \hookrightarrow M_d)$.  Since $\Omega_d$ is the $d^\text{th}$ syzygy of $\mathcal K$, we know that $\Omega_d = R(-t_d)$.  It follows that $\dim_k(M_d)_{t_d} = 1$ and $\dim_k(M_d)_{<t_d} = 0$. We denote $M_d$ by $M_{\ol y}$ and $t_d$ by $t$.

		Property \eqref{thm:countSS2.5} is a straightforward consequence of \eqref{thm:countSS2}.  As there is a unique element of minimal minimal degree in $M_{\ol y}$, say $Re = (M_{\ol y})_t$, it must be contained in a unique indecomposable summand $N_{\ol y}$ of $M_{\ol y}$.  Therefore $(N_{\ol y})_t = (M_{\ol y})_t$.

		To prove \eqref{thm:countSS3}, let $\ol{\mathcal F} = \mathcal F\otimes R/(\lst{x}{d})R$ and consider $\Tor_d^R(R/(\lst{x}{d}),R/I_{\ol y})$.  Since $R$ is \CMp, any system of parameters is a regular sequence.  Therefore,
		\[
			\Tor_d^R(R/(\lst{x}{d}),R/I_{\ol y}) \simeq H_d(\lst{x}{d}; R/I_{\ol{y}}), 
		\]	
	where $H_d(\lst{x}{d}; R/I_{\ol{y}})$ is the $d^{th}$ homology of the Koszul complex of the homogeneous system of parameters $\lst{x}{d}$ with values in $R/I_{\ol{y}}$.  Since the $x_i$ annihilate $R/I_{\ol{y}}$, we have
	\[
		H_d(x_1,\ldots,x_d; R/I_{\ol{y}}) \simeq R/I_{\ol{y}}(-t).
	\]
	Apply $\cdot \otimes_R R/(\lst{x}{d})$ to the short exact sequence 
	\[
		\xymatrixrowsep{5mm}
		\xymatrixcolsep{8mm}
		\xymatrix
			{
				 0 \ar[r] & M_{\ol y} \ar[r] & F_{d-1} \ar[r] & M_{d-1} \ar[r] & 0
			}
	\]
	to get
	\[
		\xymatrixrowsep{5mm}
		\xymatrixcolsep{8mm}
		\xymatrix
			{
				0 \ar[r] &  \Tor_1(R/(\lst{x}{d}),M_{d-1}) \ar[r] & M_{\ol y}/(\lst{x}{d})M_{\ol y} \ar[r] & \ol{F}_{d-1}.
			}
	\]
	Since $\Tor_1(R/(\lst{x}{d}),M_{d-1}) \simeq \Tor_d(R/\lst{x}{d},R/I_{\ol y})$, we have
	\[
		\xymatrix
			{
				0 \ar[r] &  R/I_{\ol{y}}(-t) \ar[r] & M_{\ol y}/(\lst{x}{d})M_{\ol y} \ar[r]^-\alpha & \ol{F}_{d-1}.
			}
	\]
	Since $e\in M_{\ol y}$ corresponds to the generator of $K_d$, it is clear that $\ol{e}\mapsto 0$, and it follows from the exact sequence that $R\ol{e} \simeq R/I_{\ol{y}}(-t)$.
	\end{proof}

	\begin{rem}
		It is worth noting that we are able to lift the restriction of the isomorphism classes in the definition of graded countable \CM type by way of Proposition \ref{prop:countable-type}.  Thus Theorem \ref{thm:countSS} holds if we have countably many graded indecomposable modules up to isomorphism in $\MCM R$.  This is a much stronger statement as we do not require graded isomorphisms.
	\end{rem}

	\begin{cor}
		Let $(R,\m,k)$ be a standard graded \CM ring with uncountable residue field $k$.  If $R$ is of graded finite \CM type then it is \SuSp.
	\end{cor}

	\begin{cor}\label{cor:countSS}
		Let $(R,\m,k)$ be a standard graded \CM ring with graded countable \CM type.  Then the possible $h$-vectors are $(1)$, $(1,n)$, or $(1,n,1)$ for some integer $n$.  
	\end{cor}
	\begin{proof}
		Combine Theorem \ref{thm:countSS} and Corollary \ref{cor:SShvector}.
	\end{proof}

\subsection{Minimal Multiplicity and Graded Countable Type}\label{sec:min-mult}

		In \cite{eisenbud88}, D. Eisenbud and J. Herzog showed that standard graded rings of graded finite \CM type and $\dim(R) > 1$ have minimal multiplicity.  Using Theorem \ref{thm:countSS}, we are able to extend this result to non-Gorenstein rings of graded countable \CM type with $\dim(R) >2$.

	In order to prove Proposition \ref{prop:dim3-min-mult}, we need a countable version of M. Auslander and I. Reiten's result: rings of graded finite \CM type have an isolated singularity \cite[Proposition 4]{auslander89}.  In \cite[Theorem 1.3]{huneke03}, C. Huneke and G. Leuschke show that the singular locus of a local ring of countable \CM type has dimension at most one.  The proof of the graded version of \cite[Theorem 1.3]{huneke03} can be found in \cite[Theorem 2.5.9]{mythesis} and is essentially a graded analog of the local version.

	\begin{prop}\label{prop:dim3-min-mult}
		Let $(R,\m,k)$ be a standard graded \CM ring of graded countable \CM type that is not Gorenstein and $\dim R \gs 3$.  Then $R$ must be a domain and have minimal multiplicity. 
	\end{prop}
	\begin{proof}
		Since $R$ is of graded countable \CM type, we know that the dimension of the singular locus is at most one.  Since $\dim(R) \gs 3$, we have that $R$ satisfies Serre's condition $(R_1)$.  Further, as $R$ is \CMp, we know that $R$ also satisfies Serre's condition $(S_2)$.  Thus by Serre's criterion, $R$ must be normal.  By normality, we can write it as a finite direct product of integrally closed domains \cite[Lemma 2.1.15]{swanson}.  As $R$ is standard graded, we have that $R_0 = k$ and thus there is only one term in the direct product. Hence $R$ is also a domain. 

		By Corollary \ref{cor:countSS}, we know that $R$ is either of minimal multiplicity or has $h$-vector $(1,n,1)$ for some positive integer $n$.  Since $R$ is not Gorenstein, there must be a socle element in degree one.  However, \cite[Theorem B]{eisenbud88} forces $R$ to have minimal multiplicity.
	\end{proof}

	\begin{rem}\label{rem:dim3-min-mult}
		It is worth noting that standard graded \CM rings of countable \CM type and dimension at least 3 are normal domains.  Even though Proposition \ref{prop:dim3-min-mult} assumed the ring was not Gorenstein, the argument to show that the ring was a normal domain still holds.
	\end{rem}

	\section{Gorenstein Rings of Graded Countable Type}

	It turns out that there are a few instances when stretched and \SuS coincide.  In particular, this happens when the ring in question is zero-dimensional or when the ring is a complete intersection that is not a hypersurface.  

	For zero-dimensional ring the empty set is a system of parameters.  Hence, we see that the expression $\sum \deg(x_i)-d+2$ in the definition of \SuS becomes just 2.  From here it is easy to see that the two definitions are equivalent.  A little more work is needed to see this for complete intersections.

	\begin{prop}\label{prop:ss-ci}
		Let $(R,\m,k)$ be a standard graded complete intersection that is stretched with $k$ an infinite field.  Then $R$ is a hypersurface or defined by two quadrics.
	\end{prop}
	\begin{proof}
		Let $S = k[\lst y n]$ and $R = S/(\lst f m)$ with $\dim(R) = d$ and $\deg(f_i) = d_i$.  Further, let $\x = (\lst x d)$ be a minimal reduction of the maximal ideal $\m$.  Given that $R$ is a complete intersection, we know that the Hilbert series of $R/\x$ is
		\begin{align}
			H_{\frac{S}{(\lst f m, \lst x d)}}(t) &= \frac{(1-t^{d_1})\cdots(1-t^{d_m})\cdot(1-t)^d}{(1-t)^n} \nonumber \\
				&= (1+t+\cdots+t^{d_1-1})\cdots(1+t+\cdots+t^{d_m-1}). \label{eq:hfun}
		\end{align}
	As $R$ is Gorenstein and stretched, we know that the $h$-vector is of the form $(1,1)$ or $(1,N,1)$ for some $N > 0$.  It is enough to only consider the $h$-vector $(1,N,1)$.  

		If the $h$-vector is $(1,N,1)$, then \eqref{eq:hfun} is of the form $1 + Nt + t^2$.  In particular, the only case to consider is when $N = 2$.  In this case, \eqref{eq:hfun} is $(1+t)(1+t)$ and thus the ideal $I$ is generated by two quadrics.
	\end{proof}

		Example \ref{ex:stretched-ss} showed that a hypersurface can be stretched but not \SuSp.  As it turns out, this is not the case when the ring is a complete intersection defined by 2 quadrics.  Recall that for a \CM ring $R$ with $h$-vector 
		\[
		(h_{\ol R}(0), h_{\ol R}(1), \ldots, h_{\ol R}(s)),
		\]
	the \dfn{socle degree} of $R$ is defined to be $\text{SocDeg}(R) = s$.

	\begin{cor}\label{cor:ss-ci}
		Let $(R,\m,k)$ be a standard graded complete intersection that is not a hypersurface. Then $R$ is stretched if and only if $R$ is \SuSp.
	\end{cor}

	\begin{proof}
		It is enough to show that stretched implies \SuSp.  To do this we show that the socle degree of $R$ modulo a homogeneous system of parameters is not too large.  Since $R$ is not a hypersurface, Proposition \ref{prop:ss-ci} implies that $R = k[\lst y n]/(f_1,f_2)$, $d = \dim(R) = n-2$, and $\deg(f_i) = 2$.  Let $\x = (\lst x d)$ be an ideal generated by a homogeneous system of parameters. As the $h$-vector of $R$ is $(1,2,1)$, we have that the socle degree of $R/\x$ is 
		\[
			\text{SocDeg}(R/\x) = \deg(f_1)+\deg(f_2) + \sum \deg{x_j} - (2+d) = \sum \deg{x_j} -d +2.
		\]
	Thus for $i \gs \sum \deg{x_j} -d +2$, 
	\[
		\dim_k\left(\frac{R}{\x}\right)_i = 
		\begin{cases}
			1 & \text{if } i = \sum \deg{x_j} -d +2, \\
			0 & \text{if } i > \sum \deg{x_j} -d +2
		\end{cases}
	\]
	for any homogeneous system of parameters $\x$ of $R$.  Therefore $R$ is \SuS as well.
	\end{proof}

	\subsection{One Dimensional Rings}

	In an effort to show the one dimensional graded case of Conjecture \ref{conj:gor}, we apply our results to rings of dimension one.  In doing so, we obtain some nice ring structure.

		\begin{thm}\label{thm:1n1}
			Let $(R,\m,k)$ be a standard graded one-dimensional \CM ring with uncountable residue field $k$.  If $R$ is of graded countable \CM type, then $R$ is a hypersurface or of minimal multiplicity.
		\end{thm}
		\begin{proof} 
			Suppose $R$ is not of minimal multiplicity and let $x \in \m$ be a homogeneous minimal reduction of the maximal ideal $\m$.  If we further assume $R$ is not a hypersurface, then we know that $h_{R/xR}(1) \gs 2$.	Further, since $R$ does not have minimal multiplicity, we know $x\m \not= \m^2$.  So let $a,b \in \m$ be linearly independent elements (modulo $x$) of a minimal generating set of $\m$ such that $a^2 \notin x\m$ or $ab \notin x\m$.  Notice that any ideal of the form $(x,a+\alpha b)$, where $\alpha \in k$, is a graded indecomposable maximal \CM module.  By Proposition \ref{prop:countable-type}, it is enough to show there are uncountably many such ideals up to isomorphism in $\MCMgr R$.  

		Consider the ideals $I_\alpha := (x,a + \alpha b)$ and $I_\beta := (x,a + \beta b)$ where $\alpha,\beta \in k$ and view them as objects in $\MCMgr R$.  Let $\varphi$ be an isomorphism between $I_\alpha$ and $I_\beta$ in $\MCMgr R$. As such, $\varphi$ is a degree zero map
			\[
				I_\alpha = (x,a + \alpha b) \stackrel{\varphi}{\simeq} (x,a + \beta b) = I_\beta
			\]
		given by
		\[
			\xymatrixrowsep{5mm}
			\xymatrixcolsep{10mm}
			\xymatrix
				{
					 x \ar@{|-{>}}[r] & d_1 x + d_2 (a + \beta b) \\
					 a + \alpha b \ar@{|-{>}}[r] & d_3 x + d_4 (a + \beta b).
				}
		\]
	Hence we see that the $d_i$'s are elements of $k$ for $i = 1,2,3,4$. Consider the relation  
		\[
			x\varphi(a + \alpha b) - (a + \alpha b)\varphi(x) = 0.
		\]
		Hence we have
		\begin{equation}\label{eqn:IdealRelation}
			d_3 x^2 + d_4 x (a + \beta b) - d_1 x (a + \alpha b) - d_2(a^2 +(\alpha + \beta)ab + \alpha\beta b^2) = 0.
		\end{equation}
		From here we can focus on $d_2$.  If $d_2 = 0$, then we have the relation
		\begin{equation}\label{eqn:IdealRelation-d=0}
			d_3 x^2 + d_4 x (a + \beta b) - d_1 x (a + \alpha b) =0.
		\end{equation}
	Since $x$ is a non-zero divisor, we can cancel $x$ and rearrange \eqref{eqn:IdealRelation-d=0} as a $k$-linear combination of $x,a,b$
		\[
			d_3 x + (d_4 -d_1)a   + (\beta d_4 - \alpha d_1)b = 0.
		\]
	As $x,a,b$ are independent over $k$, we have that the coefficients are zero.  In particular $d_4 - d_1 = 0$. Since $\varphi$ is an isomorphism, $d_1d_4-d_2d_3 \neq 0$, which forces $d_1 = d_4 \neq 0$. Hence $\beta d_4 - \alpha d_1 = 0$ implies that $\alpha = \beta$, and there are uncountably many ideals $I_\alpha$ up to isomorphism in $\MCMgr R$. A contradiction.

		If we assume that $d_2 \not= 0$, then \eqref{eqn:IdealRelation} modulo $x\m$, shows that
		\[
			a^2 +(\alpha + \beta)ab + \alpha\beta b^2 \equiv 0.
		\]
		If $a^2 \notin x\m$, notice that $R_2 = (a^2,xR_1)$.  Thus, there exist fixed $\gamma, \sigma \in k$ such that modulo $x\m$ we have
		\begin{align*}
			ab & \equiv \gamma a^2; \\
			b^2 &\equiv \sigma a^2.
		\end{align*}
		Therefore
		\begin{equation}\label{eqn:a2}
			a^2\cdot(1+\gamma(\alpha + \beta) + \sigma\alpha\beta) \equiv 0 \mod{x\m}.
		\end{equation}
		As $a^2$ is non-zero modulo $x\m$ and $1+\gamma(\alpha + \beta) + \sigma\alpha\beta$ is a degree zero element, the grading forces
		\begin{equation}\label{eqn:alpha-beta-relation}
			1+\gamma(\alpha + \beta) + \sigma\alpha\beta = 0
		\end{equation}
		in the field $k$. Choose $\alpha \in k$ such that the set $\Lambda_\alpha = \{ \beta \in k \vl I_\alpha \simeq I_\beta \}$ is uncountable.  By \eqref{eqn:alpha-beta-relation} we know that any pair $(\alpha', \beta') \in \Lambda_\alpha^2$ is a root of 
		\[
			f(X,Y) = 1+\gamma (X + Y) + \sigma XY \in k[X,Y].
		\]
		In particular, if we set $X = \alpha$, we see that $\{(\alpha,\beta)\}_{\beta\in\Lambda_\alpha}$ is an uncountable family of distinct roots of $f(X,Y)$.  This forces $f(X,Y)$ to be identically zero, a contradiction as $f(0,0) \neq 0$.

		Similarly, if $ab \notin x\m$ then there exists a fixed $\gamma', \sigma' \in k$ such that modulo $x\m$ we have
		\begin{align*}
			a^2 & \equiv \gamma' ab; \\
			b^2 &\equiv \sigma' ab.
		\end{align*}
		Therefore
		\[
			ab\cdot(\gamma'+(\alpha + \beta) + \sigma'\alpha\beta) \equiv 0 \mod{x\m}
		\]
		and we recover a similar contradiction as we did from Equation \eqref{eqn:a2}. 
		\end{proof}

		Applying this Theorem \ref{thm:1n1} to rings of graded countable \CM type brings to light some very useful structure.  In particular, if the ring is Gorenstein, then we have a hypersurface! We now have the graded, one-dimensional case of Conjecture \ref{conj:gor}.

	\begin{cor}\label{cor:1n1gor}
		Let $(R,\m,k)$ be a standard graded one dimensional Gorenstein ring with uncountable residue field $k$.  If $R$ is of graded countable \CM type, then $R$ is a hypersurface ring.
	\end{cor}
	\begin{proof}
		By Theorem \ref{thm:1n1}, $R$ is either a hypersurface or of minimal multiplicity (or both).  By Corollary \ref{cor:countSS} and the fact that $R$ is Gorenstein, we know that the possible $h$-vectors are $(1)$, $(1,1)$, or $(1,n,1)$.  Thus if $R$ has minimal multiplicity, then $R$ must also be a hypersurface.  
	\end{proof}

\section{Acknowledgements}
	
	I would like to thank the anonymous referee for valuable feedback and suggestions regarding the content of this paper. Further, I would especially like to thank Craig Huneke for several useful conversations and inspiration for the project. I am also thankful to Alessandro De Stefani for helpful discussions during the preparation of this material.  Additionally, calculations in this note were inspired by many Macaulay2 \cite{M2} computations.



\begin{thebibliography}{10}

\bibitem{arnold74}
V.~I. Arnol{\cprime}d.
\newblock Critical points of smooth functions.
\newblock In {\em Proceedings of the {I}nternational {C}ongress of
  {M}athematicians ({V}ancouver, {B}. {C}., 1974), {V}ol. 1}, pages 19--39.
  Canad. Math. Congress, Montreal, Que., 1975.

\bibitem{auslander86}
Maurice Auslander.
\newblock Isolated singularities and existence of almost split sequences.
\newblock In {\em Representation theory, {II} ({O}ttawa, {O}nt., 1984)}, volume
  1178 of {\em Lecture Notes in Math.}, pages 194--242. Springer, Berlin, 1986.

\bibitem{auslander89}
Maurice Auslander and Idun Reiten.
\newblock Cohen-{M}acaulay modules for graded {C}ohen-{M}acaulay rings and
  their completions.
\newblock In {\em Commutative algebra ({B}erkeley, {CA}, 1987)}, volume~15 of
  {\em Math. Sci. Res. Inst. Publ.}, pages 21--31. Springer, New York, 1989.

\bibitem{auslander89b}
Maurice Auslander and Idun Reiten.
\newblock The {C}ohen-{M}acaulay type of {C}ohen-{M}acaulay rings.
\newblock {\em Adv. in Math.}, 73(1):1--23, 1989.

\bibitem{buchweitz87}
R.-O. Buchweitz, G.-M. Greuel, and F.-O. Schreyer.
\newblock Cohen-{M}acaulay modules on hypersurface singularities. {II}.
\newblock {\em Invent. Math.}, 88(1):165--182, 1987.

\bibitem{eisenbud88}
David Eisenbud and J{\"u}rgen Herzog.
\newblock The classification of homogeneous {C}ohen-{M}acaulay rings of finite
  representation type.
\newblock {\em Math. Ann.}, 280(2):347--352, 1988.

\bibitem{fouli2011}
Louiza Fouli and Craig Huneke.
\newblock What is a system of parameters?
\newblock {\em Proc. Amer. Math. Soc.}, 139(8):2681--2696, 2011.

\bibitem{M2}
Daniel~R. Grayson and Michael~E. Stillman.
\newblock Macaulay2, a software system for research in algebraic geometry.
\newblock Available at \href{http://www.math.uiuc.edu/Macaulay2/}%
  {http://www.math.uiuc.edu/Macaulay2/}.

\bibitem{herzog78}
J{\"u}rgen Herzog.
\newblock Ringe mit nur endlich vielen {I}somorphieklassen von maximalen,
  unzerlegbaren {C}ohen-{M}acaulay-{M}oduln.
\newblock {\em Math. Ann.}, 233(1):21--34, 1978.

\bibitem{huneke03}
Craig Huneke and Graham~J. Leuschke.
\newblock Local rings of countable {C}ohen-{M}acaulay type.
\newblock {\em Proc. Amer. Math. Soc.}, 131(10):3003--3007 (electronic), 2003.

\bibitem{knorrer87}
Horst Kn{\"o}rrer.
\newblock Cohen-{M}acaulay modules on hypersurface singularities. {I}.
\newblock {\em Invent. Math.}, 88(1):153--164, 1987.

\bibitem{leuschke}
Graham~J. Leuschke and Roger Wiegand.
\newblock {\em Cohen-{M}acaulay representations}, volume 181 of {\em
  Mathematical Surveys and Monographs}.
\newblock American Mathematical Society, Providence, RI, 2012.

\bibitem{northcott54}
D.~G. Northcott and D.~Rees.
\newblock Reductions of ideals in local rings.
\newblock {\em Proc. Cambridge Philos. Soc.}, 50:145--158, 1954.

\bibitem{sally79}
Judith~D. Sally.
\newblock Stretched {G}orenstein rings.
\newblock {\em J. London Math. Soc. (2)}, 20(1):19--26, 1979.

\bibitem{schreyer87}
Frank-Olaf Schreyer.
\newblock Finite and countable {CM}-representation type.
\newblock In {\em Singularities, representation of algebras, and vector bundles
  ({L}ambrecht, 1985)}, volume 1273 of {\em Lecture Notes in Math.}, pages
  9--34. Springer, Berlin, 1987.

\bibitem{mythesis}
Branden Stone.
\newblock {\em Super-stretched and graded countable {C}ohen-{M}acaulay type}.
\newblock 2012.
\newblock Thesis (Ph.D.)--University of Kansas.

\bibitem{swanson}
Irena Swanson and Craig Huneke.
\newblock {\em Integral closure of ideals, rings, and modules}, volume 336 of
  {\em London Mathematical Society Lecture Note Series}.
\newblock Cambridge University Press, Cambridge, 2006.

\bibitem{tovpyha12}
Oleksii Tovpyha.
\newblock Graded {C}ohen-{M}acaulay rings of wild representation type.
\newblock {\em arXiv.org}, math.AC, Dec 2012.
\newblock arXiv:1212.6557v1 [math.AG].

\bibitem{yoshino90}
Yuji Yoshino.
\newblock {\em Cohen-{M}acaulay modules over {C}ohen-{M}acaulay rings}, volume
  146 of {\em London Mathematical Society Lecture Note Series}.
\newblock Cambridge University Press, Cambridge, 1990.

\end{thebibliography}
	
\def\cprime{$'$}

\end{document}